\documentclass[a4paper,twoside,8pt]{article}
\usepackage[T1]{fontenc}
\usepackage[latin1]{inputenc}
\usepackage{graphicx}
\usepackage{amsmath}
\usepackage{amsthm}
\usepackage{amsfonts}
\usepackage{amssymb}
\usepackage{epsfig}
\usepackage{enumerate}
\usepackage{bm,bbm}
\usepackage{fullpage}
\usepackage{mathtools}
\usepackage{varioref}
\usepackage{float}
\newtheorem{theorem}{\textbf{Theorem}}[section]
\newtheorem{lemma}{\textbf{Lemma}}[section]
\newtheorem{proposition}{\textbf{Proposition}}[section]
\newtheorem{definition}{\textbf{Definition}}[section]
\newtheorem{corollary}{\textbf{Corollary}}[section]
\newtheorem{remark}{\textbf{Remark}}[section]

\def\E{\mathcal{E}}

\renewcommand{\div}{ \operatorname{div} }
\DeclareMathOperator*{\argmin}{{\rm argmin}}

\renewenvironment{proof}{\textbf{Proof. }}{\hfill$\Box$
\\}

\title{A non linear approximation method for solving high dimensional partial differential equations: Application in Finance.}
\author{Jos\'e Infante Acevedo and Tony Leli\`evre\\
  Universit\'e Paris-Est, CERMICS,\\ Ecole des Ponts,   6-8 avenue
Blaise Pascal,  77455
Marne-la-vall\'{e}e, France\\  {\tt jose-infante.acevedo@cermics.enpc.fr, \ lelievre@cermics.enpc.fr}}

\begin{document}  
\maketitle
{\noindent {\bf Abstract:} We study an algorithm which has been proposed
  in~\cite{ammar-mokdad-chinesta-keunings-06, nouy-07} to solve
  high-dimensional partial differential equations. The idea is to represent
  the solution as a sum of tensor products and to compute iteratively the
  terms of this sum. This algorithm is related to the so-called greedy
  algorithm introduced by Temlyakov in~\cite{temlyakov-08}.
  In this paper, we investigate the application of the greedy algorithm in
  finance and more precisely to the option pricing problem. We approximate the
  solution to the Black-Scholes equation and we propose a
  variance reduction method. In numerical experiments, we obtain results for
  up to $10$ underlyings. Besides, the proposed variance reduction method
  permits an important reduction of the variance in comparison
with a classical Monte Carlo method.}\\

 
{\noindent {\bf Acknowledgements:} Jos\'e Infante Acevedo is grateful to AXA Research~Fund for his doctoral fellowship.}

\section*{Introduction}

Many problems of interest for various applications (for example material
sciences and finance) involve high-dimensional partial differential equations
(PDEs). The typical example in finance is the pricing of a basket option,
which can be obtained by solving the Black-Scholes PDE with dimension the number of underlying assets.

We propose to investigate an algorithm which has been recently proposed by
Chinesta~{\it et al.}~\cite{ammar-mokdad-chinesta-keunings-06} for solving
high-dimensional Fokker-Planck equations in the context of kinetic models for
polymers, and by Nouy~{\it et al.}~\cite{nouy-12} in uncertainty
quantification framework based on previous works by
Ladev\`eze~\cite{ladeveze-99}. This approach is also studied
in~\cite{le-bris-lelievre-maday-09} to try to circumvent the curse of
dimensionality for the Poisson problem in high-dimension. This approach is a
nonlinear approximation method called below the {\it greedy
  algorithm} because it is related to the so-called greedy algorithms
introduced in nonlinear approximation theory, see for
example~\cite{temlyakov-08}. The main idea is to represent the solution as a
sum of tensor products (referred to as a {\em separated representation} in the following): 

\begin{eqnarray}
  \label{eq:sol_pgd}
  u(x_1,\ldots ,x_d)&=&\sum_{k\geq 1}r^1_k(x_1)r^2_k(x_2)\ldots r^d_k(x_d) \nonumber \\
&=&\sum_{k\geq 1}\left(r^1_k\otimes r^2_k\otimes\ldots\otimes r^d_k\right)(x_1,\ldots ,x_d)
\end{eqnarray}

\noindent and to compute iteratively each term of this sum using a greedy
algorithm. This greedy algorithm can be applied to any PDE which admits a
variational interpretation as a minimization problem. The practical interest
of this algorithm has been demonstrated in various
contexts (see for example~\cite{ammar-mokdad-chinesta-keunings-06} for applications in fluid mechanics).

Our contribution is to complete the first application of this algorithm in
finance, investigating the interest of this approach for option pricing. In
this work, our aim is to study the problem of pricing vanilla basket options
of European type using two numerical methods: first a discretization technique
for the Black-Scholes PDE and a variance reduction method for the pricing of the same type of financial products.

For option pricing, we will discuss in particular the key points to be solved to address problems
in finance, compared to the situations studied
in~\cite{le-bris-lelievre-maday-09} or
in~\cite{ammar-mokdad-chinesta-keunings-06}, that is, the treatment of the non
zero boundary conditions and the approximation of the solution to the
Black-Scholes PDE as a sequence of minimization problems. We will study also the practical implementation of the algorithm. We will not
solve the minimization problems associated to the PDE, but the first-order
optimality conditions of these minimization problems, namely the Euler
equation. This leads to a system of equations where the number of degrees of
freedom does not grow exponentially with respect to the dimension, and this
fact will be very important in order to attain high-dimensional frameworks in
practical applications. More precisely, the Euler equation writes as a system
of $d$ nonlinear equations, where $d$ is the considered dimension. The maximum
dimension that can be treated by this technique is limited by the
{\em non-linearity} of the system of $d$ equations that has to be solved.

The variance reduction method relies on the backward Kolmogorov equation which
yields an exact control variate. We propose to solve the high-dimensional
Kolmogorov equation using the greedy algorithm. This yields an efficient
pricing method which combines deterministic and stochastic techniques.

We would like to point out that in order to circumvent the curse of
dimensionality using the greedy algorithm, the initial condition of the
Black-Scholes equation has to be set out in separated representation with
respect to the different coordinates, namely a sum of tensor products. As the initial condition is not always expressed in a
separated representation, we first need to investigate the problem of
approximating the initial condition by a sum of tensor products. This
problem will be solved using again the greedy algorithm. We will provide examples to
illustrate that this approximation is suitable.

Other deterministic techniques have been applied to solve the Black-Scholes PDE in a
high-dimensional framework. Classical methods such as finite differences and
finite elements are limited in their application when the dimension increases
(typically $d\leq 4$), because the number of degrees of freedom increases exponentially with respect
to the dimension and rapidly exceed the limited storage capacity. Financial
applications of the sparse tensor product methods have been studied by Pommier
in~\cite{pommier-08}. These sparse methods also use the representation of the
solution as a sum of tensor products, and assume that the solution is regular
enough to obviate fine discretizations in each direction. In practice, this
method may be difficult to apply for reasons such as the lack of regularity of
the solution and the difficulty to implement it.

In our numerical experiments, the greedy algorithm gives results for up to $10$ underlyings; that
means the dimension $d$ is equal to $10$. To the best of our knowledge, this
is higher than results obtained using other deterministic approaches, such as
the tensor product method, for which examples up to the dimension
$d=5$ have been reported in the literature, see~\cite{pommier-08}. In addition, the variance reduction method that we are
proposing permits the variance to be reduced in comparison with a classic Monte Carlo method. 

As future research, we have in mind to generalize the method to American options. This has a theoretical counterpart, namely the
generalization of greedy algorithms to free boundary problems that has already
been analyzed in~\cite{cances-ehrlacher-lelievre-10}.

The plan of the paper is the following. In
Section~\ref{section:presentation_greedy_algo}, we introduce the general
setting for the greedy algorithm and we give some
theoretical results that have been proved in the literature and that ensure
the convergence of the greedy
algorithm. Section~\ref{section:implementation_algorithm} will discuss the practical implementation of the
greedy algorithm in the case of the approximation of a function by a sum of
tensor products. Following this, in
Section~\ref{section_initial_condition_svd}
and~\ref{section:pricing_basket_with_payoff}, we will present results and
applications for the approximation of a basket put option. The purpose of
Section~\ref{section:black_scholes_pde_gral} is to apply the greedy algorithm
to solve the Black-Scholes PDE. After introducing the weak formulation
of this equation in Section~\ref{section:black_scholes_pde}, we treat the difficulties that arise when applying the greedy algorithm such
as posing the problem in a bounded domain
(Section~\ref{subsection:bounded_domain}) and recasting the PDE as a
minimization problem (Section~\ref{subsection:imex_min_pb}). The final
section (Section~\ref{section:numerical_results}) contains numerical results
for the solution of the Black-Scholes PDE and for the variance reduction method.

\section{Greedy algorithms for high dimensional problems}
\label{section:presentation_greedy_algo}

In this section, we define a general framework for the greedy algorithm that
we will use to solve the high-dimensional problems studied in this paper. 

The bottom line of deterministic approaches for high-dimensional problems is to represent the solutions as
linear combinations of tensor products of one-dimensional functions as
in~\eqref{eq:sol_pgd}. If the number of terms in the expansion remains small,
this enables us to approximate efficiently the solution, avoiding the curse of dimensionality.

The greedy algorithm proposed
in~\cite{ammar-mokdad-chinesta-keunings-06,le-bris-lelievre-maday-09,nouy-12}
is based on two important points. The first one is that we need to recast the original problem (in our case, this is the option pricing problem) as a minimization problem:

\begin{equation}
\label{eq:min_pb}
  u = \argmin_{v \in V} \E(v),
\end{equation}

\noindent where $\E:V\mapsto\mathbb{R}$ is a functional with a unique global
minimizer $u\in V$ with $V$ a Hilbert space. For example, for approximating
the solution to a Poisson problem with homogeneous Dirichlet boundary
conditions, one would consider
$\mathcal{E}(v)=\frac{1}{2}\int_{\mathcal{X}}\left|\nabla
  v\right|^2-\int_{\mathcal{X}}fv$ and $V=H^1_0(\mathcal{X}^d)$ with
$\mathcal{X}$ a bounded one-dimensional domain and $d$ large (see~\cite{le-bris-lelievre-maday-09}).

The second point is to look iteratively for the best tensor product in the
expansion of the solution as a sum of tensor products of lower-dimensional functions

\begin{equation}
  \label{eq:greedy_sol}
  u_n(x_1,x_2,\ldots ,x_d)=\sum^n_{k=1}r^1_k\otimes r^2_k\ldots\otimes r^d_k(x_1,\ldots ,x_d)
\end{equation}

\noindent where for all $i=1,\ldots ,d$ and $k=1,\ldots ,n$, the functions $r^i_k\in
V_{x_i}$, with $V_{x_i}$ Hilbert spaces of functions depending on the one-dimensional variable $x_i$. This
sequential search for the terms in the sum~\eqref{eq:greedy_sol} is related to the so-called greedy algorithms introduced in
the nonlinear approximation theory by Temlyakov in~\cite{temlyakov-08} and by
De Vore and Temlyakov in~\cite{devore-temlyakov-96}.

In what follows, we assume that $V,V_{x_1},V_{x_2},\ldots ,V_{x_d}$ are
Hilbert spaces such that
\begin{itemize}
\item[(H1)] $\mbox{Vect} \{ r^1 \otimes r^2\otimes ,\ldots\otimes r^d, r^1 \in
  V_{x_1}, r^2 \in V_{x_2},\ldots ,r^d\in V_{x_d}\} \subset V$ is dense.
\end{itemize}

To compute $u_n$ in the separated form~\eqref{eq:greedy_sol}, $u_n$ being the
approximation of $u$ solution of the problem~\eqref{eq:min_pb}, we define the greedy algorithm
as follows,

Iterating for all $n\geq 1$:

\begin{equation}
  \label{def_greedy_algo}
(r^1_n,r^2_n,\ldots ,r^d_n) \in \argmin_{r^1 \in V_{x_1},\,r^2 \in
  V_{x_2},\ldots ,r^d\in V_{x_d}} \E\left( \sum_{k=1}^{n-1} r^1_k \otimes
  r^2_k \otimes\ldots\otimes r^d_k + r^1 \otimes r^2\otimes\ldots\otimes r^d \right).
\end{equation}

The following result given in~\cite{le-bris-lelievre-maday-09} ensures the
convergence of the greedy algorithm defined in~\eqref{def_greedy_algo} for a functional $\E$ of the following form

\begin{equation}
  \label{eq:functional_E_form}
\E(v)=\|u-v\|^2_V.  
\end{equation}

\begin{theorem}
  Let us assume that the assumption (H1) holds and that the functional $\E$
  has the form~\eqref{eq:functional_E_form}. Then,
  \begin{equation}
    \label{eq:cv_greedy}
    \| u_n-u \|_V\underset{n\rightarrow\infty}\longrightarrow 0.
  \end{equation}
\end{theorem}

An estimate of the error is also proposed
in~\cite{le-bris-lelievre-maday-09}. To state this result, we need to
introduce the functional space adapted to the convergence analysis

\begin{equation}
  \label{eq:definition_set_cv}
  \mathcal{A}^1=\left\{u=\sum^{+\infty}_{k=1}r^1_k\otimes
    r^2_k\otimes\ldots\otimes r^d_k,\text{ where }r^i_k\in V_{x_i},i=1,\ldots
    ,d,\sum^{+\infty}_{k=1}\|r^1_k\otimes r^2_k\ldots\otimes r^d_k\|_V <+\infty\right\}
\end{equation}

and the associated norm which is

\begin{equation}
  \label{eq:norm_set_cv}
  \|u\|_{\mathcal{A}^1}=\text{inf}\left\{ \sum^{+\infty}_{k=1}\|r^1_k\otimes
    r^2_k\ldots\otimes r^d_k\|_V \mid u=\sum^{+\infty}_{k=1}r^1_k\otimes
    r^2_k\otimes\ldots\otimes r^d_k \right\}.
\end{equation}

\begin{theorem}
  Let us assume that the assumption (H1) is verified and that the functional $\E$
  has the form $\eqref{eq:functional_E_form}$. Then, for a function $u\in \mathcal{A}^1$, there exists a constant $C>0$ such that
  \begin{equation}
    \label{eq:spped_cv_greedy}
    \| u_n-u \|_V\leq Cn^{-1/6},
  \end{equation}
  for all $n\in\mathbb{N}^*$.
\end{theorem}

We note that the convergence rate factor of $\frac{1}{6}$ can be improved to
$\frac{11}{62}$ and that the constant $C$ depends on the norm
$\|u\|_{\mathcal{A}^1}$, see~\cite{le-bris-lelievre-maday-09,devore-temlyakov-96}. 

Our work relies on these theoretical results because in the setting of this paper, we will consider a functional $\E$ and a
Hilbert space $V$ such that $\E(v)=\|u-v\|^2_V$, (see equation~\eqref{eq:energy_i} below). 

\section{Implementation of the algorithm in the case of the approximation of a square-integrable function}

In this section, we will discuss the implementation of the algorithm defined
by~\eqref{def_greedy_algo} in the case of the approximation of a given function $f$
by a sum of tensor products. We will then provide numerical examples of this approach. This particular case has the advantage of being
an easy example to understand the implementation of the greedy
algorithm. Moreover, this procedure is useful in a preliminary step to
approximate the initial condition of the Black-Scholes PDE (see
Section~\ref{section:black_scholes_pde_gral}), namely in order to get a separated representation of the payoff function.

\subsection{Greedy algorithm for the approximation of a square-integrable function}
\label{section:implementation_algorithm}

In order to show the implementation that we use for the
algorithm~\eqref{def_greedy_algo}, let us present the simple problem of
approximating a square-integrable function $f$ by a sum of tensor
products. Mathematically, we consider the spaces $V=L^2(\Omega_1 \times
\Omega_2\times\ldots\times\Omega_d)$, $V_{x_i}=L^2(\Omega_i)$ for $i=1,\ldots ,d$, where
$\Omega_i\subset\mathbb{R}$ is a bounded domain for $i$ such that $1\leq i\leq
d$. We recall that we are looking for a separated representation $f=\sum_{k\geq
  1}r^1_k\otimes r^2_k\otimes\ldots\otimes r^d_k$. Let us consider the following minimization problem:

\begin{equation}
  \label{eq:min_pb_separated_rep}
  \text{Find }u\in V\text{ such that }u=\text{arg }\text{min}_{v\in V}\left(\frac{1}{2}\int_{\Omega_1\times\Omega_2\times\ldots\times\Omega_d}v^2-\int_{\Omega_1\times\Omega_2\times\ldots\times\Omega_d}vf\right)
\end{equation}

\noindent whose solution is obviously $u=f$. In this context, the greedy
algorithm~\eqref{def_greedy_algo} can be rewritten as follows:

Iterate for all $n\geq 1$: Find $(r^1_n,r^2_n,\ldots
,r^d_n)\in V_{x_1}\times V_{x_2}\times\ldots\times V_{x_d}$ such that
$(r^1_n,r^2_n,\ldots ,r^d_n)$ belongs to

\begin{equation}
  \label{eq:greedy_algo_svd}
  \begin{split}
  \argmin_{r^1 \in V_{x_1},\ldots ,r^d \in V_{x_d}}\;\;
  &\frac{1}{2}\int_{\Omega_1 \times \Omega_2\times\ldots\times\Omega_d}\left|\sum^{n-1}_{k=1}r^1_k\otimes
  r^2_k\otimes\ldots\otimes r^d_k+r^1\otimes r^2\otimes\ldots\otimes r^d\right| ^2  \\
&-\int_{\Omega_1 \times
  \Omega_2\times\ldots\times\Omega_d}\left(\sum^{n-1}_{k=1}r^1_k\otimes
  r^2_k\otimes\ldots r^d_k+r^1\otimes r^2\otimes\ldots\otimes r^d\right)f.
\end{split}
\end{equation}

As proposed in~\cite{le-bris-lelievre-maday-09}, instead of solving the
problem~\eqref{eq:greedy_algo_svd}, we will determine the solutions of the Euler
equation for~\eqref{eq:greedy_algo_svd}. Notice that,
in general, the solutions of the Euler equation are not necessarily the
solutions of the minimization problem, given the nonlinearity of the tensor
product space $L^2(\Omega_1)\otimes L^2(\Omega_2)\otimes\ldots\otimes L^2(\Omega_d)$.

The Euler equation for~\eqref{eq:greedy_algo_svd} has the following form:

Find $(r^1_n,r^2_n,\ldots ,r^d_n)\in V_{x_1}\times
V_{x_2}\times\ldots\times V_{x_d}$ such that for any functions
$(r^1,r^2,\ldots ,r^d)\in V_{x_1}\times V_{x_2}\times\ldots\times V_{x_d}$

\begin{equation}
\label{el_svd_2D}
\begin{split}
\int_{\Omega_1\times\Omega_2\times\ldots\times\Omega_d}(r^1_n\otimes
r^2_n\otimes\ldots\otimes r^d_n)\left(r^1\otimes r^2_n\otimes\ldots\otimes
  r^d_n+r^1_n\otimes r^2\otimes\ldots\otimes r^d_n+\ldots +r^1_n\otimes
  r^2_n\otimes\ldots\otimes
  r^d\right) \\
=\int_{\Omega_1\times\Omega_2\times\ldots\times\Omega_d}f_{n-1}\left(r^1\otimes r^2_n\otimes\ldots\otimes
  r^d_n+r^1_n\otimes r^2\otimes\ldots\otimes r^d_n+\ldots +r^1_n\otimes
  r^2_n\otimes\ldots\otimes
  r^d\right)&
\end{split}
\end{equation}
\noindent where $f_{n-1}=f-\sum^{n-1}_{k=1}r^1_k\otimes r^2_k\otimes\ldots\otimes r^d_k$.

Equation~\eqref{el_svd_2D} can be written equivalently as 

\begin{align}
  \label{e_l_svd_system}
  \begin{cases}
    \|r^2_n\|^2\|r^3_n\|^2\ldots\|r^d_n\|^2\; r^1_n = \int_{\Omega_2\times\Omega_3\times\ldots\times\Omega_d}
    \left(r^2_n\otimes\ldots\otimes r^d_n\right)\ f_{n-1},& \\
    \|r^1_n\|^2\|r^3_n\|^2\|r^4_n\|^2\ldots\|r^d_n\|^2\; r^2_n =\int_{\Omega_1\times\Omega_3\times\Omega_4\ldots\times\Omega_d} \left(r^1_n\otimes
      r^3_n\otimes r^4_n\otimes\ldots\otimes r^d_n\right)\ f_{n-1},& \\
     \;\;\;\;\vdots&   \\
      \|r^1_n\|^2\|r^2_n\|^2\|r^3_n\|^2\ldots\|r^{d-1}_n\|^2\; r^d_n =\int_{\Omega_1\times\Omega_2\times\Omega_3\ldots\times\Omega_{d-1}} \left(r^1_n\otimes
      r^2_n\otimes r^3_n\otimes\ldots\otimes r^{d-1}_n\right)\ f_{n-1},&
  \end{cases}
\end{align}

\noindent where $\|r^i_n\|^2$ denotes the square $L^2$-norm: $\|r^i_n\|^2=\int_{\Omega_i}|r^i_n|^2$.

The system~\eqref{e_l_svd_system} is a non linear coupled system of equations
on which a fixed point procedure can be applied as proposed
in~\cite{ammar-mokdad-chinesta-keunings-06}. Choose $(r^{1,(0)}_n,r^{2,(0)}_n,\ldots ,r^{d,(0)})\in
L^2(\Omega_1)\times L^2(\Omega_2)\times\ldots\times L^2(\Omega_d)$, and at iteration $k\geq 0$, compute $(r^{1,(k)}_n,r^{2,(k)}_n,\ldots ,r^{d,(k)})\in
L^2(\Omega_1)\times L^2(\Omega_2)\times\ldots\times L^2(\Omega_d)$ which is
the solution to

\footnotesize
\begin{align}
  \label{e_l_svd_system_fixed_point}
  \begin{cases}
    \|r^{2,(k)}_n\|^2\|r^{3,(k)}_n\|^2\ldots\|r^{d,(k)}_n\|^2\; r^{1,(k+1)}_n=\int_{\Omega_2\times\Omega_3\times\ldots\times\Omega_d}
    \left(r^{2,(k)}_n\otimes\ldots\otimes r^{d,(k)}_n\right)\ f_{n-1},& \\
    \|r^{1,(k+1)}_n\|^2\|r^{3,(k)}_n\|^2\|r^{4,(k)}_n\|^2\ldots\|r^{d,(k)}_n\|^2\;
    r^{2,(k+1)}_n=\int_{\Omega_1\times\Omega_3\times\Omega_4\ldots\times\Omega_d}\left(r^{1,(k+1)}_n\otimes
      r^{3,(k)}_n\otimes r^{4,(k)}_n\otimes\ldots\otimes r^{d,(k)}_n\right)\ f_{n-1},& \\
     \;\;\;\;\vdots&   \\
      \|r^{1,(k+1)}_n\|^2\|r^{2,(k+1)}_n|^2\|r^{3,(k+1)}_n\|^2\ldots\|r^{d-1,(k+1)}_n\|^2\;
    r^{d,(k+1)}_n & \\ 
    =\int_{\Omega_1\times\Omega_2\times\Omega_3\ldots\times\Omega_{d-1}} \left(r^{1,(k+1)}_n\otimes
      r^{2,(k+1)}_n\otimes r^{3,(k+1)}_n\otimes\ldots\otimes
      r^{d-1,(k+1)}_n\right)\ f_{n-1},&
  \end{cases}
\end{align}
\normalsize

\noindent until convergence is reached.

An important point to note is that we start with a linear
problem~\eqref{eq:min_pb_separated_rep} with exponential complexity with
respect to the dimension, and at the end, we obtain a nonlinear
problem~\eqref{e_l_svd_system} with at each iteration linear complexity with
respect to the dimension. This is a general feature of the greedy algorithm:
the curse of dimensionality is circumvented, but the linearity of the original
problem is lost because the space of tensor products is non-linear. 

In the two-dimensional case, the algorithm given by~\eqref{eq:greedy_algo_svd}
is related to the Singular Value Decomposition (or rank one decomposition), as it is explained in~\cite{le-bris-lelievre-maday-09}. In this case, the solutions of the variational problem~\eqref{eq:greedy_algo_svd} are exactly the solutions to the Euler equation~\eqref{el_svd_2D} that verify the second-order optimality conditions. This property does not hold in a $d$-dimensional framework with $d\geq 3$. 

\subsection{Example of a separated representation of a put payoff}
\label{section_initial_condition_svd}

In this section we will apply the algorithm~\eqref{eq:greedy_algo_svd} to
obtain an approximation of the payoff of a basket put option. For the
practical implementation of the greedy algorithm, we need to introduce the
space discretization. In practice, the spaces $V^{\Delta x}_{x_i}$ for
$i=1,\ldots ,d$ that are used to discretize $V_{x_i}=L^2(\Omega)$ with $\Omega_i=(0,1)$ for $i=1,\ldots d$
are the P$1$ finite elements on a uniform mesh with space step
$\Delta x$. The integer $N=\frac{1}{\Delta x}$ is the number of intervals in each direction. For each $k$ and $i=1,\ldots ,d$, we discretize the functions
$r^i_k$ that appear in the approximation of the solution
given by the expression~\eqref{el_svd_2D} as follows:

\begin{equation}
\label{notation_r_s}
  r^{i,\Delta x}_k(x_i)=\sum^N_{j=0}r^{i,j}_{k}\phi_j(x_i),\;\;\;r^{i,j}_{k}\in\mathbb{R},\;\forall j,k,
\end{equation}

\noindent where $\phi_i(x)=\phi\left(\frac{x-x_i}{\Delta x}\right)$ with $\phi(x)=\left\{\begin{array}{ll}1-|x| & \text{if }|x|\leq 1, \\ 0 & \text{if }|x|>0. \end{array}\right. $  
\\
\\

\noindent This type of discretization and its generalization to the
$d$-dimensional case will be used for all the numerical simulations of this paper.

Let us now consider the problem~\eqref{eq:greedy_algo_svd} with
$f(x_1,\ldots ,x_d)=\left(K-\frac{1}{d}\sum_{i=1}^d x_i\right)_+$. Figure~\ref{5D_put_it_1} shows how the algorithm approximates the basket put payoff in a two-dimensional framework ($d=2$). We observe that, as the number of iterations of the greedy algorithm increases, the approximation of the function $f(x_1,x_2)$ improves.

\begin{figure}[htbp]
 \centering
 \includegraphics[width=0.5\linewidth]{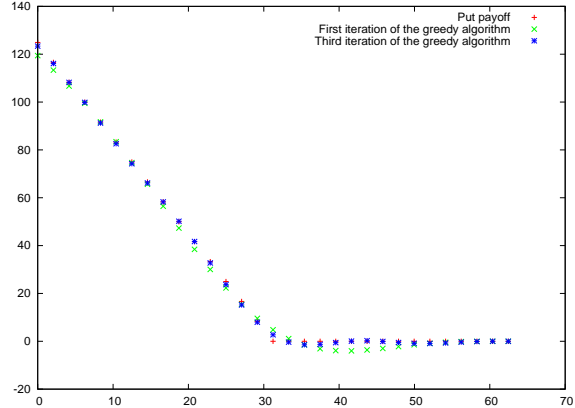} 
 \caption{Basket put option with two assets. We consider here the intersection between the surface of prices and the plane
 $S_1=S_2$. To obtain this approximation we
   take 31 points of discretization per dimension (N= 30). In this figure, we show
   the approximation given after the first and third iteration of the algorithm.}
\label{5D_put_it_1}
\end{figure}

Figure~\ref{cv_curves_put_payoff} shows the convergence curves that we obtain
for this problem as the dimension $d$ increases. We observe that, as the dimension increases, the number of iterations needed to obtain the convergence increases as well.

\begin{figure}[htbp]
 \centering
 \includegraphics[width=0.5\linewidth]{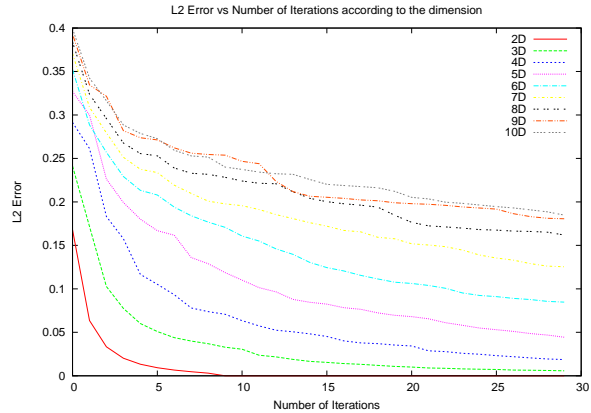} 
 \caption{Convergence curves for the approximation of a basket put payoff by a
   sum of tensor products. We observe that if the dimension increases, then
   the number of iterations needed to obtain the convergence increases as
   well. The error calculated is given by the equation
   ~\eqref{eq:relative_error} below.}
\label{cv_curves_put_payoff}
\end{figure}

We also provide in Table~\ref{table_nb_it_cv_payoff} the number of
iterations needed in order to obtain a relative error of $10^{-5}$ when we
consider 11 points of discretization per dimension ($N=10$). The relative error
calculated is the discrete $L^2$ error

\begin{equation}
\label{eq:relative_error}  
  e_n=\frac{\sqrt{\frac{1}{N}\sum^N_{i_1=1}\sum^N_{i_2=1}\ldots\sum^N_{i_d=1}\left(f(x_{i_1},x_{i_2},\ldots,x_{i_d})-u^n(x_{i_1},x_{i_2},\ldots,x_{i_d})\right)^2}}{\sqrt{\frac{1}{N}\sum^N_{i_1=1}\sum^N_{i_2=1}\ldots\sum^N_{i_d=1}f(x_1,\ldots
      ,x_d)^2}}
\end{equation}

\noindent where $u^n(x_1,x_2,\ldots ,x_d)=\sum^n_{k=1}r^1_k\otimes
r^2_k\otimes\ldots\otimes r^d_k$ is the solution obtained with the greedy
algorithm at the iteration $n$. It is to be outlined that computing the norm is more costly than the greedy algorithm itself.

\begin{table}
\begin{center}
\begin{tabular}{|c|c|}
  \hline
  Dimension & Number of iterations \\
  \hline
  1 & 1 \\
  2 & 2 \\
  3 & 10 \\
  4 & 22 \\
  5 & 101 \\
  6 & 228 \\
  7 & 1077 \\
  8 & 3974 \\
  \hline
\end{tabular}
\caption{Number of iterations needed to obtain a relative error of $10^{-5}$
  when we take 11 discretization points per dimension\label{table_nb_it_cv_payoff}}
\end{center}
\end{table}

Notice that the full tensor product approximation would require $11^8
\simeq 2.10^8$ degrees of freedom in an $8$-dimensional case to be compared
with the $3974\times 8\times 11\simeq 350000$ degrees of freedom that we
obtained. For the evaluation of this number, we used the fact that at each iteration of the algorithm we get $8$ functions that are determined by $11$ discretization points.

In order to reduce the computational time of this calculation, we use the
specific form of the payoff function to deduce, in a preliminary step, the
points that belong to the support of this function. Therefore, when we
calculate numerically the integral term

\begin{equation}
\label{integral_term_euler}  
\int_{\Omega_2\times\Omega_3\times\ldots\times\Omega_d} \left(r^2_n\otimes\ldots\otimes
  r^d_n\right)\ f\;dx_2dx_3\ldots dx_d
\end{equation}

\noindent in~\eqref{e_l_svd_system}, we do
not need to pass through the points where the function $f$ vanishes. In practice,
we have used a backtracking algorithm to describe the support of the payoff function. This type of
algorithm consists in constructing candidates sequentially and neglecting them
when they do not verify the conditions required as a solution, in this case to
belong to the support of the payoff function. For instance, in a $5$-dimensional case,
the computational time is reduced by a factor of $\frac{4}{5}$ by taking into
account the support of the payoff function in the computations of the integral term.

Let us now make a few remarks.

\begin{remark}

In our numerical experiments, the initial conditions needed to begin the
iterations in the fixed point procedure are taken randomly. We indeed observed
from numerical experiments that it yields better results in terms of
convergence than in the case where constant initial conditions are used.

\end{remark}

\begin{remark}

Concerning the computational time, we note that if the dimension increases, one iteration of the algorithm takes
more time to be computed because the number of equations in the Euler
system~\eqref{e_l_svd_system} increases linearly with respect to the dimension. The integral terms of type~\eqref{integral_term_euler} also demand more time of execution because the domain has a new variable.

\end{remark}

\subsection{Pricing of a basket put using the separated approximation of the payoff}
\label{section:pricing_basket_with_payoff}

As an example to show that the approximation by a sum of tensor
products makes sense, we can use this approximation as a method to obtain
prices of options in the Black-Scholes framework. The price of a European
option in the Black-Scholes model (see Section~\ref{section:black_scholes_pde}) is given by

\begin{equation}
  \label{eq:pricing_svd}
  \begin{split}
    P_t & = \mathbb{E}\left[ e^{-r(T-t)} f(S_1(T),\ldots ,S_d(T))|\mathcal{F}_t
    \right] \\
    & = \int_{\Omega_1\times\ldots\times\Omega_d}e^{-r(T-t)}f(y_1,\ldots
    ,y_d)g(T,y_1,\ldots ,y_d|t,S_1(t),\ldots ,S_d(t))dy_1\ldots dy_d
  \end{split}
\end{equation}

\noindent where $(\mathcal{F}_t)_{t\geq 0}$ is the natural filtration generated by the $d$ assets prices $S_i(t)\; (i=1,\ldots ,d)$, $f$ is the payoff of the option, $g(T,. | t,S_1(t),\ldots ,S_d(t))$ is the joint
density of the variables $S_1(T),\ldots ,S_d(T)$ given the values $(S_1(t),\ldots ,S_d(t))$ of the underlying assets at time $t$. This joint density is a log-normal law and thus has an explicit analytical expression. 

Using the greedy algorithm as we saw in the previous section, we can obtain a
separable approximation of the product

$$f(y_1,\ldots ,y_d)g(T,y_1,\ldots ,y_d|t,S_1(t),\ldots ,S_d(t)),$$

\noindent and thus the integral~\eqref{eq:pricing_svd} can be calculated very
efficiently using the Fubini's rule. 

\noindent In Figure~\ref{7D_price_put_integral}, we apply this idea for the case of a
basket put option on seven assets.

\begin{figure}[htbp]
 \centering
 \includegraphics[width=0.5\linewidth]{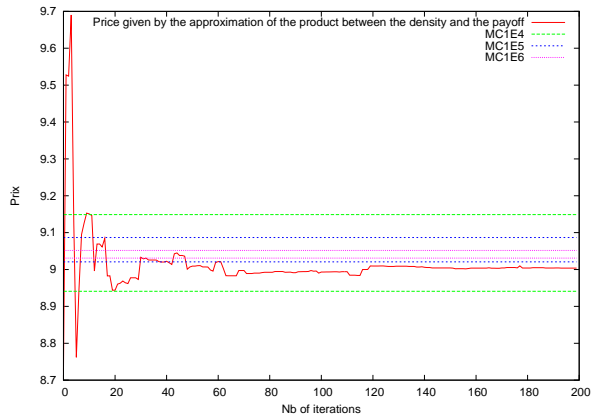} 
 \caption{Price of put basket option with $7$ assets. The continuous curve gives
   the price of this financial product with respect to the number of
   iterations of the algorithm. The horizontal lines represent the confidence
   interval obtained with a Monte Carlo method using respectively $10^4$, $10^5$ and $10^6$ iterations.}
\label{7D_price_put_integral}
\end{figure}

\section{Greedy algorithm for solving the Black-Scholes partial differential equation}
\label{section:black_scholes_pde_gral}

Now, we will apply the greedy algorithm~\eqref{def_greedy_algo} presented in
Section~\ref{section:presentation_greedy_algo} to solve the Black-Scholes
equation, and obtain the price of a European option.

\subsection{Weak formulation of the Black-Scholes partial differential equation}   
\label{section:black_scholes_pde}

The Black-Scholes model in a $d$-dimensional framework describes the dynamics
of $d$ risky assets that satisfy the following stochastic differential
equations:

\begin{equation}
\label{dyn_bs}
  \frac{dS_i(t)}{S_i(t)}=rdt+\sigma_idB_i(t)\text{ for all }i=1,\ldots ,d,
\end{equation}

with

\begin{equation}
\label{correl_browniens}
d\langle B_i,B_j\rangle_t=\rho_{ij}dt.
\end{equation}

The number $\rho_{ij}$ is the correlation between the Brownian motions $B_i$ and
$B_j$ that drive the dynamics of the assets $S_i$ and $S_j$
respectively.

The coefficient $\sigma_i$ represents the volatility of the asset $S_i$ at
time t, and $r$ is the risk-free instantaneous interest rate. To simplify, we
assume that $r$ and $\sigma_i$ for $i=1,\ldots d$ are constant during the
period $[0,T]$. We note that the greedy algorithm that we are proposing can be used
when the risk-free interest rate is a continuous function of time and the
volatility is a continuous function of time and of the asset under standard regularity assumptions (See Chapter 2 in~\cite{achdou-pironneau-05}).

We recall that the price of a European option with payoff $f$ and maturity $T$
is given by the following formula

$$\mathbb{E}\left[e^{r(T-t)}f(S_1(T),\ldots ,S_d(T))|\mathcal{F}_t\right].$$ 

Using the Markovianity of the process
$(S_1(t),\ldots ,S_d(t))$, it can be written as: 

\begin{equation}
  \label{eq:feynman_kac}
  \mathbb{E}\left[e^{r(T-t)}f(S_1(T),\ldots ,S_d(T))|\mathcal{F}_t\right]=P(t,S_1(t),\ldots S_d(t)),
\end{equation}

\noindent where $P$ is a deterministic function.

The function $P(t,S_1,\ldots ,S_d)$ satisfies the Black-Scholes PDE which can be obtained using the Feynman-Kac theorem. The Black-Scholes equation in a $d$-dimensional framework is a parabolic PDE that has the following form:

\begin{eqnarray}
\label{eq:bs_edp_S}
\begin{cases}
  \frac{\partial P}{\partial t}+\mathcal{L}P=0,\;\;\; t<T, (S_1,\ldots ,S_d)\in\mathbb{R}^d_+, \\
  P(T,S_1,\ldots ,S_d)=f(S_1,\ldots ,S_d),\;\;\;(S_1,\ldots ,S_d)\in\mathbb{R}^d_+, 
\end{cases}
\end{eqnarray}
where the operator $\mathcal{L}$ is given by 
\begin{equation*}
  \mathcal{L}P=\frac{1}{2}\sum^d_{i,j=1}\frac{\partial^2  P}{\partial S_i\partial S_j}\rho_{ij}\sigma_i\sigma_jS_iS_j+\sum^d_{i=1}rS_i\frac{\partial P}{\partial S_i}-rP.
\end{equation*}

Let us recall the standard framework for problem~\eqref{eq:bs_edp_S}. Setting $\tau:=T-t$, the time to maturity, we get the following forward parabolic problem for $\hat{P}(\tau,S_1\ldots ,S_d)=P(t,S_1\ldots ,S_d)$

\begin{eqnarray}
  \label{eq:std_bs_fwd}
  \begin{cases}
    \frac{\partial \hat{P}}{\partial \tau}-\mathcal{L}\hat{P}=0, \;\;\; 0<\tau\leq T, (S_1,\ldots ,S_d)\in\mathbb{R}^d_+ \\
\hat{P}(0,S_1,\ldots ,S_d)=f(S_1,\ldots ,S_d), \;\;\; (S_1,\ldots ,S_d)\in\mathbb{R}^d_+.
  \end{cases}
\end{eqnarray}

We note that it is possible to write the diffusion term in the operator $\mathcal{L}+r$ in a divergence form as follows:

\begin{equation*}
  \mathcal{L}\hat{P}+r\hat{P}=\frac{1}{2}\sum^d_{i=1}\frac{\partial }{\partial S_i}\left(\sum^d_{j=1}\rho_{i,j}\sigma_i\sigma_jS_iS_j\frac{\partial \hat{P}}{\partial S_j}\right)+\sum^d_{j=1}\left(rS_j-\frac{1}{2}\sum^d_{i=1}\frac{\partial}{\partial S_i}(\rho_{i,j}\sigma_i\sigma_jS_iS_j)\right)\frac{\partial \hat{P}}{\partial S_j}.
\end{equation*}

Therefore, if we multiply $-\mathcal{L}\hat{P}$ by a test function $Q$ and then we integrate on $\mathbb{R}^d_+$, we obtain the following bilinear form:

\begin{equation}
\label{bilinear_form_bs_before_chg_variables}
\begin{split}
  b_t(\hat{P},Q) & =
  \frac{1}{2}\sum^d_{i=1}\sum^d_{j=1}\int_{\mathbb{R}^d_+}\rho_{i,j}\sigma_i\sigma_jS_iS_j\frac{\partial
    \hat{P}}{\partial S_j}\frac{\partial Q}{\partial S_i} \\
   &\;\;\;\; -
   \sum^d_{j=1}\int_{\mathbb{R}^d_+}\left(rS_j-\frac{1}{2}\sum^d_{i=1}\frac{\partial}{\partial
       S_i}(\rho_{i,j}\sigma_i\sigma_jS_iS_j)\right)\frac{\partial \hat{P}}{\partial
     S_j}Q+r\int_{\mathbb{R}^d_+}\hat{P}Q.
 \end{split}
\end{equation}

Let us now introduce the Hilbert space 

\begin{equation*}
  \mathcal{V}(\mathbb{R}^d_+)=\left\{ v:v\in L^2(\mathbb{R}^d_+),S_i\frac{\partial v}{\partial S_i}\in L^2(\mathbb{R}^d_+), i=1,\ldots ,d \right\}
\end{equation*}

and its norm

\begin{equation*}
  \|v\|_{\mathcal{V}}=\left(\|v\|^2_{L^2(\mathbb{R}^d_+)}+\sum^d_{i=1}\left\|S_i\frac{\partial v}{\partial S_i}\right\|^2_{L^2(\mathbb{R}^d_+)}\right)^\frac{1}{2}.
\end{equation*}

We have the following result for problem~\eqref{eq:std_bs_fwd} (see
Theorem 2.11 in~\cite{achdou-pironneau-05}).

\begin{theorem}
\label{uniq_exis_fv_bs}
Let us assume that the matrix $\Xi$ defined by
$\Xi_{i,j}=\rho_{i,j}\sigma_i\sigma_j$ is positive-definite. Then for all $f\in L^2(\mathbb{R}^d_+)$, there exists a unique function $\hat{P}\in L^2(0,T;\mathcal{V})\cap\mathcal{C}^0([0,T];L^2(\mathbb{R}^d_+))$, with $\frac{\partial \hat{P}}{\partial t}\in L^2(0,T;\mathcal{V}')$ such that, for any function $\phi\in \mathcal{D}(0,T)$, for all $v\in\mathcal{V}$,

\begin{equation}
\label{fv_bs}
-\int^T_0\phi'(t)\left(\int_{\mathbb{R}^d_+}\hat{P}(t)v\right)dt+\int^T_0\phi(t)b_t(\hat{P}(t),v)dt=0  
\end{equation}

and

\begin{equation}
\label{fv_bs_init_condit}
\hat{P}(t=0)=f.  
\end{equation}

\end{theorem}

This result shows the existence and uniqueness of a weak solution for the problem~\eqref{eq:std_bs_fwd}.

In this work, our goal is to obtain the curve of prices for a put basket option which has a square-integrable payoff. The price of call basket options can be obtained by the well-known put-call parity. So, as initial condition we consider the payoff function:

\begin{equation}
  \label{eq:init_cond}
  f(S_1,\ldots ,S_d)=\left(K-\frac{1}{d}\sum^d_{i=1}S_i(0)\right)_+
\end{equation}

\noindent where the constant $K$ is the strike of the option.

Three new difficulties appear when we want to apply the greedy
algorithm~\eqref{def_greedy_algo} to solve the problem ~\eqref{eq:std_bs_fwd}
when we compare it with the application of the greedy algorithm to the case of the Poisson problem studied in~\cite{le-bris-lelievre-maday-09}:

\begin{enumerate}
\item It is a problem posed on an infinite domain.
\item It is a time-dependent problem.
\item We cannot simply recast the weak formulation~\eqref{fv_bs} of the problem~\eqref{eq:std_bs_fwd} as a minimization problem because the bilinear form~\eqref{bilinear_form_bs_before_chg_variables} is non-symmetric.
\end{enumerate}  

\subsection{Formulation on a bounded domain}
\label{subsection:bounded_domain}

The financial assets $S_i$ for $i=1,\ldots ,d$ take values in $[0,\infty)$. Consequently, we have
to deal with an infinite domain. Let us then introduce the following transformations:

\begin{equation}
  \label{eq:chang_var_direct}
  \Psi:\mathbb{R}_+\mapsto [0,1],\;\;\; s\mapsto\frac{s}{s+\frac{K}{d}},
\end{equation}

\begin{equation}
  \label{eq:chang_var_inverse}
\Psi^{-1}:[0,1]\mapsto\mathbb{R}_+,\;\;\; x\mapsto \frac{xK}{d(1-x)}.  
\end{equation}

As remarked by Pommier in~\cite{pommier-08}, the change of
variables~\eqref{eq:chang_var_direct} maps bijectively $\mathbb{R}_+$
to the interval $[0,1]$ and appears to be efficient in practice since it leads to a
refined mesh around the singularity line of the payoff function. In~\cite{pommier-08}, Pommier explains that if we
set a classical localized boundary-domain then the volume next to this
singularity decays exponentially with the dimension. This change of variables
allows us not to impose artificial boundary conditions contrary to classical truncation techniques. Proposition~\ref{chg_variables} below shows that with the change of variables~\eqref{eq:chang_var_direct}, we get a well-posed problem in a bounded domain without boundary conditions.

Applying the change of variable~\eqref{eq:chang_var_direct} into the equation $\eqref{eq:bs_edp_S}$, we obtain:

\begin{eqnarray}
  \label{eq:op_bs_edp_x}
  \begin{cases}
    -\frac{\partial u}{\partial t}+\tilde{\mathcal{L}}u=0,& \\
    u(0,x_1,\ldots ,x_d)=(K-\frac{K}{d}\sum^d_{i=1}\frac{x_i}{1-x_i}),&
  \end{cases}
\end{eqnarray}

\noindent where $u(t,x_1,\ldots ,x_d)=P(t,S_1,\ldots ,S_d)$ with
$S_i=\Psi(x_i)$,$(x_1,\ldots ,x_d)\in\Omega=(0,1)^d$ and

\begin{equation*}
\tilde{\mathcal{L}}u=\div(A\nabla u)+\sum^d_{i=1}\left[r+\sigma^2_ix_i-\sigma^2_i+\frac{\sigma_i}{2}\sum^d_{j=1,j\neq
  i}\rho_{i,j}\sigma_j(2x_j-1)\right]x_i(1-x_i)\frac{\partial u}{\partial x_i}-ru,  
\end{equation*}

\noindent with the matrix $A$ given by

\begin{equation}
  \label{A_def}
  A_{i,j}(x_1,\ldots ,x_d):=\frac{\rho_{i,j}\sigma_i\sigma_j}{2}x_ix_j(1-x_i)(1-x_j).
\end{equation}

Then let us introduce the following Hilbert space

\begin{equation}
  \tilde{\mathcal{V}}(\Omega)=\left\{ v\in L^2(\Omega)\; |\;\forall\; 1\leq i\leq d,\; (1-x_i)x_i\frac{\partial v}{\partial x_i}\in L^2(\Omega) \right\},
\end{equation}

\noindent endowed with the norm

\begin{equation}
  \| v\|_{\tilde{\mathcal{V}}}=\left( \|v\|^2_{L^2(\Omega)}+|v|^2_{\tilde{\mathcal{V}}}\right)^{\frac{1}{2}}
\end{equation}

\noindent where

\begin{equation}
  \label{eq:semi_norm_V}
  |v|^2_{\tilde{\mathcal{V}}}=\sum^d_{i=1}\left\|x_i(1-x_i)\frac{\partial v}{\partial x_i}\right\|^2_{L^2(\Omega)}. 
\end{equation}

In what follows, we need the following lemma:

\begin{lemma}
\label{lemma:density}
 The space $\mathcal{C}^{\infty}_c(\Omega)$ is dense in $\tilde{\mathcal{V}}(\Omega)$.

\end{lemma}

This lemma can be directly deduced from Lemma 2.6 in~\cite{achdou-pironneau-05}.

\begin{corollary}
The following integration by parts formula holds:

\begin{equation}
  \label{eq:integration_by_parts}
  \int_{\Omega}\div(A\nabla {u})v=-\int_{\Omega}(A\nabla u)\nabla v, \;\;\forall u,v\in \tilde{\mathcal{V}}(\Omega).
\end{equation}

\end{corollary}

\begin{proof}
  It follows from Lemma~\ref{lemma:density}.
\end{proof}

Therefore, multiplying $-\tilde{\mathcal{L}}u$ by a test function
$v\in\mathcal{C}^{\infty}_c(\Omega)$ and then using~\eqref{eq:integration_by_parts}, we get the following bilinear form:

\begin{equation}
  \label{eq:biln_form_b_tilde}
  \tilde{b}_t(u,v)=\int_{\Omega}(A\nabla u)\nabla v-\int_{\Omega}(a\nabla u)v+\int_{\Omega}ruv.
\end{equation}

\noindent where $a=(a_1,\ldots ,a_d):\Omega\mapsto\mathbb{R}^d$ is the vector field with
$i$-th component given by

\begin{equation}
\label{a_vector_def}
a_i(x_1,\ldots ,x_d)=x_i(1-x_i)\left[r+\sigma^2_ix_i-\sigma^2_i+\frac{\sigma_i}{2}\sum^d_{j=1,j\neq  i}\rho_{ij}\sigma_j(2x_j-1)\right],  
\end{equation}

The following result holds for the the bilinear form defined in~\eqref{eq:biln_form_b_tilde}

\begin{lemma}
\label{lemma_garding}
The bilinear form $\tilde{b}_t$ is continuous from $\tilde{\mathcal{V}}\times\tilde{\mathcal{V}}$, that is, there exists a constant $\overline{c}$ that does not depend on $t$ such that for all functions $v,w\in\tilde{\mathcal{V}}$
\begin{equation}
  \label{eq:cont_bilin_form_b_tilde}
  \tilde{b}_t(v,w)\leq\overline{c}\|v\|_{\tilde{\mathcal{V}}}|w|_{\tilde{\mathcal{V}}}
\end{equation}
Moreover, the bilinear form $\tilde{b}_t$ verifies a Garding inequality, that is, there exist two positive constants $\underline{c}>0$ and $\lambda>0$ such that for all functions $v\in\tilde{\mathcal{V}}$
\begin{equation}
  \label{eq:garding_biln_b_tilde}
  b(v,v)\geq\underline{c}|v|^2_{\tilde{\mathcal{V}}}-\lambda\|v\|^2_{L^2(\Omega)}
\end{equation}
\end{lemma}

\begin{proof}
  The Garding inequality is obtained by observing that the first order term
  satisfies the following:
  \begin{equation*}
    \left|\int_{\Omega}p(x_1,\ldots ,x_d)\frac{\partial v}{\partial x_i}v \right|\leq\frac{1}{2}\left|\int_{\Omega}\frac{\partial p}{\partial x_i}(x_1,\ldots ,x_d)v^2 \right|,
  \end{equation*}
where $p(x_1,\ldots ,x_d)$ is a polynomial. The proof of the continuity uses the same arguments.  
\end{proof}

Thus, we obtain the following result concerning the existence and uniqueness of the solution for the weak formulation associated to the problem~\eqref{eq:op_bs_edp_x}.

\begin{proposition}
\label{chg_variables}  
For all function $g\in L^2(\Omega)$, there exists a unique $u\in L^2(0,T;\tilde{\mathcal{V}})\cap\mathcal{C}^0([0,T];L^2(\Omega))$, with $\frac{\partial u}{\partial t}\in L^2(0,T;\tilde{\mathcal{V}}')$ such that for any function $\phi\in\mathcal{D}(0,T)$, for all function $v\in\tilde{\mathcal{V}}$,
\begin{equation}
  \label{fv_with_chg_variables}
  -\int^T_{0}\phi'(t)\int_{\Omega}u(t)vdt+\int^T_{0}\phi(t)\tilde{b}(u,v)dt=0,
\end{equation}
and
\begin{equation}
  u(t=0)=g
\end{equation}

Moreover, $u$ solution of~\eqref{fv_with_chg_variables} is related to
$\hat{P}$ solution of~\eqref{fv_bs} by the functions defined in~\eqref{eq:chang_var_direct} and~\eqref{eq:chang_var_inverse}.   
\end{proposition}

This proposition can be deduced from Lemma~\ref{lemma_garding} using
standard techniques, see~\cite{achdou-pironneau-05}.

\subsection{The IMEX scheme and the Black-Scholes equation as a minimization problem}
\label{subsection:imex_min_pb}

To apply the greedy algorithm~\eqref{def_greedy_algo}, our goal is to rewrite
the problem~\eqref{eq:op_bs_edp_x} as a minimization problem. As a first step, we propose to use an Euler scheme to discretize the
problem in time. Let us consider a time discretization grid of $M+1$
points, $\tau_0=0\leq\ldots\leq\tau_M=T$, where $\tau_i=i\Delta t$ and $\Delta
t=\frac{T}{M}$. We introduce a time discretization of the variational formulation~\eqref{fv_with_chg_variables} where we treat explicitly the non-symmetric
part of $\tilde{b}_t$ and implicitly its symmetric terms (IMEX
scheme). 

For $i=1,\ldots ,M$, find $u^i\in\tilde{\mathcal{V}}$ such that

\begin{equation}
  \label{eq:bs_edp_imex}
\begin{split}
  \int_{\Omega}u^iv+\frac{\Delta t}{1+r\Delta t}\int_{\Omega}(A\nabla
  u^i)\nabla v-\frac{\Delta t}{2(1+r\Delta t)}\left[\int_{\Omega}(a\nabla u^i)v+\int_{\Omega}(a\nabla v)u^i\right] \\ 
 = \frac{1}{1+r\Delta t}\int_{\Omega}u^{i-1}v+\frac{\Delta t}{2(1+r\Delta
   t)}\left[\int_{\Omega}(a\nabla u^{i-1})v-\int_{\Omega}(a\nabla
   v)u^{i-1}\right],\;\;\forall v\in\tilde{\mathcal{V}}.
\end{split}
\end{equation}

Thus, using that the left hand side of the equation~\eqref{eq:bs_edp_imex} is
symmetric in $u^i$ and $v$, we are led to solve the following sequence of minimization problems

 For $i=1,\ldots ,M$:

\begin{equation}
  \label{eq:min_pb_bs_imex}
  \text{Find }u^i\in \tilde{\mathcal{V}}(\Omega)\text{ such that }u^i=\argmin_{u\in \tilde{\mathcal{V}}(\Omega)}\mathcal{E}_i(u)
\end{equation}

\noindent where 

\begin{equation}
  \label{eq:energy_i}
  \begin{split}
  \mathcal{E}_i(u)&=\frac{1}{2}\int_{\Omega}|u|^2+\frac{\Delta t}{2(1+r\Delta t)}\left[\int_{\Omega}(A\nabla
  u)\nabla u-\int_{\Omega}(a\nabla u)u\right]\\
&-\frac{1}{1+r\Delta
  t}\int_{\Omega}u^{i-1}u-\frac{\Delta t}{2(1+r\Delta
  t)}\left[\int_{\Omega}(a\nabla u^{i-1})u-\int_{\Omega}(a\nabla
  u)u^{i-1}\right].
\end{split}
\end{equation}

Let us introduce the bilinear symmetric form $\hat{a}(u,v)$

\begin{equation}
  \label{sym_form_hat_a}
  \hat{a}(u,v)=\int_{\Omega}uv+\frac{\Delta t}{1+r\Delta t}\int_{\Omega}(A\nabla
  u)\nabla v-\frac{\Delta t}{2(1+r\Delta t)}\left[\int_{\Omega}(a\nabla
    u)v+\int_{\Omega}(a\nabla v)u  \right],\;\;\forall u,v\in \tilde{\mathcal{V}}
\end{equation}

\noindent and the linear form

\begin{equation}
 \label{lin_form_L} 
  L_{i-1}(v)=\frac{1}{1+r\Delta t}\int_{\Omega}u^{i-1}v+\frac{\Delta t}{2(1+r\Delta
    t)}\left[\int_{\Omega}(a\nabla u^{i-1})v-\int_{\Omega}(a\nabla
    v)u^{i-1}\right],\;\;\forall v\in \tilde{\mathcal{V}}.
\end{equation}

We have that $\E_i(u)=\frac{1}{2}\hat{a}(u,u)-L_{i-1}(u)$.

\subsection{Stability analysis for the IMEX scheme}
\label{section:stability_analysis}

In this section, we study the $L^2$-stability of the IMEX
scheme~\eqref{eq:bs_edp_imex}. Let us consider $u^n$ the solution of the problem~\eqref{eq:min_pb_bs_imex} at time $\tau_n$. In our context, the meaning of stability is given in the following definition.

\begin{definition}
  The numerical scheme~\eqref{eq:bs_edp_imex} is $L^2$-stable if there exists a constant $C>0$ that does not depend on the discretization parameter $\Delta t$ such that for any initial condition $u^0$ and for all $n\geq 0$
  \begin{equation*}
    \|u^n\|_{L^2}\leq C\|u^0\|_{L^2}
  \end{equation*}
\end{definition}

We will see that the scheme~\eqref{eq:bs_edp_imex} is $L^2$-stable under a condition on the time step $\Delta t$. For the sake of
simplicity we take $r=0$. Let us begin with the following lemma:

\begin{lemma}
\label{lemma:A_bounded}  
Let us assume that the matrix $\Xi$ defined by
$\Xi_{i,j}=\rho_{i,j}\sigma_i\sigma_j,\; i,j=1,\ldots ,d$, is positive-definite. Then there exists a constant $\alpha>0$ such that  

\begin{equation}
  \label{stab_ineq}
  \int_{\Omega}\left(A\nabla u\right)\nabla
  u\geq\alpha|u|^2_{\tilde{\mathcal{V}}},\;\;\;\forall u\in\tilde{\mathcal{V}}. 
\end{equation}

\end{lemma}

\begin{proof}
  Using that the matrix $\Xi$ is positive-definite, we have

  \begin{equation*}
    \int_{\Omega}\sum^d_{i,j=1}\frac{\rho_{ij}\sigma_i\sigma_j}{2}x_i(1-x_i)\frac{\partial
  u}{\partial x_i}x_j(1-x_j)\frac{\partial u}{x_j}=\int_{\Omega}Y^T\Xi Y\geq \left(\min_{\lambda\in Sp(\Xi)}\lambda \right) |u|^2_{\tilde{\mathcal{V}}}
  \end{equation*}

  \noindent where $Sp(\Xi)$ denotes the set of eigenvalues of the matrix $\Xi$ and $Y$ is the
  vector such that $Y_i=x_i(1-x_i)\frac{\partial u}{\partial x_i}$. This proves~\eqref{stab_ineq} with $\alpha=\min_{\lambda\in Sp(\Xi)}\lambda$.   
\end{proof}

Now, we can state the following proposition:

\begin{proposition}
  \label{prop:stability}
  The scheme proposed in~\eqref{eq:bs_edp_imex} is $L^2$-stable under the following
  condition on the time step $\Delta t$

\begin{equation}
 \label{cond_stab_ineq_6}
\Delta t<\frac{1}{2\left(\frac{4\left(\|\tilde{a}\|_{\infty}+\|\div(a)\|_{\infty}\right)}{\alpha}+\frac{\alpha}{2}\right)}
\end{equation}

\noindent where the constant $\alpha$ is defined in Lemma~\ref{lemma:A_bounded} and $\tilde{a}$ is the vector such that for all $i=1,\ldots d$

\begin{equation*}
\tilde{a}_i=\left[r+\sigma^2_ix_i-\sigma^2_i+\frac{\sigma_i}{2}\sum^d_{j=1,j\neq i}\rho_{i,j}\sigma_j(2x_j-1)\right].
\end{equation*}

\end{proposition}

\begin{proof}

  Let us take $r=0$ and $v=u^i$ in the variational
  formulation~\eqref{eq:bs_edp_imex}. Thus, we obtain

  \begin{eqnarray}
 \label{stability_analysis} 
 & & \frac{1}{2\Delta t}\left(\int_{\Omega}|u^i|^2-|u^{i-1}|^2\right)+\frac{1}{2\Delta t}\int_{\Omega}|u^i-u^{i-1}|^2+\int_{\Omega}(A\nabla u^i)\nabla u^i-\int_{\Omega}(a\nabla u^i)u^i  \nonumber \\
& & = \frac{1}{2}\left[\int_{\Omega}(a\nabla u^{i-1})u^i-\int_{\Omega}(a\nabla u^i)u^{i-1} \right] \nonumber \\ 
& & \Leftrightarrow \frac{1}{2\Delta t}\left(\int_{\Omega}|u^i|^2-|u^{i-1}|^2\right)+\frac{1}{2\Delta t}\int_{\Omega}|u^i-u^{i-1}|^2+\int_{\Omega}(A\nabla u^i)\nabla u^i \nonumber \\
& & -\frac{1}{2}\int_{\Omega}(a\nabla (u^i+u^{i-1}))u^i-\frac{1}{2}\int_{\Omega}(a\nabla u^i)(u^i-u^{i-1})=0. 
\end{eqnarray}

Moreover, we note that $a_i=x_i(1-x_i)\tilde{a}_i$ and that $\tilde{a}\in L^{\infty}(\Omega)$.

Thus, for all $\epsilon>0$ we have that

\begin{equation}
\label{stab_ineq_1}
\left|\int_{\Omega}(a\nabla u^i)(u^i-u^{i-1})\right|\leq\epsilon|
u^i|^2_{\tilde{\mathcal{V}}}+\frac{\| \tilde{a}\|_{\infty}}{4\epsilon}\int_{\Omega}|u^i-u^{i-1}|^2.  
\end{equation}

Besides, using the integration by parts given
by~\eqref{eq:integration_by_parts} to study the term
$\int_{\Omega}(a\nabla(u^i+u^{i-1}))u^i$, we observe that

\begin{equation}
  \label{stab_ineq_2}
  \begin{split}
  \left|\int_{\Omega}a\nabla u^i(u^i+u^{i-1})\right|&=\left|\int_{\Omega}a\nabla
  u^{i}\left(2u^i+(u^{i-1}-u^i)\right)\right|, \\
  &\leq\epsilon|u^i|^2_{\tilde{\mathcal{V}}}+\frac{\|\tilde{a}\|_{\infty}}{4\epsilon}\left(\int_{\Omega}2|u^i-u^{i-1}|^2+\int_{\Omega}8|u^i|^2\right)
  \end{split}
\end{equation}

\noindent and

\begin{equation}
  \label{stab_ineq_3}
\left|\int_{\Omega}\div(a)u^i(u^i+u^{i-1})\right|\leq\frac{\|\div(a)\|_{\infty}}{4\epsilon}\left(\int_{\Omega}2|u^i-u^{i-1}|^2+\int_{\Omega}8|u^i|^2\right)+\epsilon\int_{\Omega}|u^i|^2. 
\end{equation}

Then, using~\eqref{stab_ineq_1},~\eqref{stab_ineq_2} and~\eqref{stab_ineq_3}, we
deduce from~\eqref{stability_analysis} that

\begin{equation}
  \label{stab_ineq_4}
  \begin{split}
  &\left[\frac{1}{2\Delta
      t}-\frac{\|\tilde{a}\|_{\infty}}{\epsilon}-\frac{\|\div(a)\|_{\infty}}{\epsilon}-\frac{\epsilon}{2}\right]\int_{\Omega}|u^i|^2+(\alpha-\epsilon)|u^i|^2_{\tilde{\mathcal{V}}} \\
  &+\left[\frac{1}{2\Delta
      t}-\frac{\|\tilde{a}\|_{\infty}}{4\epsilon}-\frac{\|\div(a)\|_{\infty}}{4\epsilon}\right]\int_{\Omega}|u^i-u^{i-1}|^2\leq\frac{1}{2\Delta
  t}\int_{\Omega}|u^{i-1}|^2.
  \end{split}
\end{equation}

So, if we assume the condition~\eqref{cond_stab_ineq_6} and we choose $\epsilon=\frac{\alpha}{2}$, then we can deduce the following
three inequalities:

\begin{equation}
  \alpha>\epsilon,\; \frac{1}{2\Delta
  t}-\frac{\|\tilde{a}\|_{\infty}+\|\div(a)\|_{\infty}}{\epsilon}-\frac{\epsilon}{2}>0\text{ and }\frac{1}{2\Delta t}>\frac{\|\tilde{a}\|_{\infty}+\|\div(a)\|_{\infty}}{4\epsilon}. 
\end{equation}

Consequently, we get

\begin{equation*}
  \begin{split}
  \int_{\Omega}|u^i|^2&\leq\frac{1}{1-C\Delta t}\int_{\Omega}|u^{i-1}|^2\\
  &\leq (1+2C\Delta t)\int_{\Omega}|u^{i-1}|^2 \\
  &\leq (1+2C\Delta t)^M\int_{\Omega}|u^{0}|^2 \\
  &\leq e^{2CT}\int_{\Omega}|u^{0}|^2, \\
  \end{split}
\end{equation*}

\noindent where $C=2\left(\frac{(\|\tilde{a}\|_{\infty}+\|\div(a)\|_{\infty})}{\epsilon}+\frac{\epsilon}{2}\right)$ is a constant which is independent of the discretization parameter~$\Delta t$.
\end{proof}

\subsection{Implementation of the greedy algorithm for the Black-Scholes PDE}
\label{subsection:implementation_greedy_algo}

To simplify the notation we consider the case of only three dimensions, but
the definition of the algorithm and all the equations below can be easily generalized to a $d$-dimensional framework.

We recall that the greedy algorithm will generate the function $u^i$ in the
separated representation:

\begin{equation*}
  \label{eq:u_i}
  u^{i}(x_1,x_2,x_3)=\sum_{k\geq 1}r^{i}_k\otimes s^{i}_k\otimes t^{i}_k(x_1,x_2,x_3).
\end{equation*}

The greedy algorithm~\eqref{eq:min_pb_bs_imex} is defined as follows: For $i=1,\ldots ,M$, iterate on $n\geq 1$

\begin{equation}
\label{min_pb_bs_imex_greedy}
(r^i_n,s^i_n,t^i_n) \in\argmin_{\begin{array}{l} r\in \tilde{\mathcal{V}}(\Omega_1),\\
      s\in \tilde{\mathcal{V}}(\Omega_2),\\
    t\in \tilde{\mathcal{V}}(\Omega_3)\end{array}} \frac{1}{2}\hat{a}(r\otimes s\otimes t,r\otimes s\otimes t)-L_{i-1}(r\otimes s\otimes t)-\hat{a}\left(\sum^{n-1}_{k=1}r^i_k\otimes s^i_k\otimes t^i_k,r\otimes s\otimes t\right)
\end{equation}

\noindent where $\hat{a}$ is defined by~\eqref{sym_form_hat_a} and $L_{i-1}$ by~\eqref{lin_form_L}.

Then, the Euler equation associated with the problem~\eqref{min_pb_bs_imex_greedy}, that is used in practice to implement the algorithm, is given by:

Find $(r^i_n,s^i_n,t^i_n)\in\tilde{\mathcal{V}}(\Omega_1)\times\tilde{\mathcal{V}}(\Omega_2)\times\tilde{\mathcal{V}}(\Omega_3)$ such that for any functions $(r,s,t)\in\tilde{\mathcal{V}}(\Omega_1)\times\tilde{\mathcal{V}}(\Omega_2)\times\tilde{\mathcal{V}}(\Omega_3)$
\small
\begin{equation}
  \label{bs_euler_lagrange}
  \begin{split}
\hat{a}(r^i_n\otimes s^i_n\otimes t^i_n,r\otimes s^i_n\otimes
t^i_n+r^i_n\otimes s\otimes t^i_n+r^i_n\otimes s^i_n\otimes t)&=
L_{i-1}(r\otimes s^i_n\otimes t^i_n+r^i_n\otimes s\otimes t^i_n+r^i_n\otimes
s^i_n\otimes t) \\
+\hat{a}\left(r\otimes s^i_n\otimes t^i_n+r^i_n\otimes s\otimes t^i_n+r^i_n\otimes
s^i_n\otimes t,\sum^{n-1}_{k=1}r^i_k\otimes s^i_k\otimes t^i_k\right)&
\end{split}
\end{equation}
\normalsize
Henceforth, we will consider without loss of generality $s=0$ and $t=0$ in order to study in detail each term of this
Euler equation~\eqref{bs_euler_lagrange}. We recall that the Euler equation is
solved using a fixed point procedure as in~\eqref{e_l_svd_system_fixed_point}.

\begin{remark}

All the high-dimensional integrals in~\eqref{bs_euler_lagrange} are easily
calculated using Fubini's rule because the functions in these integrals are separable except for
$i=1$ where the term $u^0$ appears as follows:

\begin{equation}
\label{high_dim_integral}
\int_{\Omega_1\times\Omega_2\times\Omega_3}r\otimes s^i_n\otimes t^i_n(x_1,x_2,x_3)u^0(x_1,x_2,x_3)dx_1,dx_2dx_3,
\end{equation}

The idea used to overcome this practical obstacle is to approximate, in a
preliminary step, the initial condition $u^0$ of the Black-Scholes PDE by a sum of
tensor products as explained in Section~\ref{section_initial_condition_svd}. Once the initial condition $u^0$ has a separated approximation:

\begin{equation*}
  \label{eq:u_0}
  u^{0}(x_1,x_2,x_3)=\sum_{k\geq 1}r^{0}_k\otimes s^{0}_k\otimes t^{0}_k(x_1,x_2,x_3).
\end{equation*}

\noindent the integral~\eqref{high_dim_integral} is easy to compute using Fubini's rule.
\end{remark}

Using the space discretization described in Section~\ref{section_initial_condition_svd} and the notation given by~\eqref{notation_r_s}, the following vectors
will be used:

\begin{equation*}
\mathbf{r^i_n}=[r^i_{n,0},\ldots ,r^i_{n,N}]^T,\;\;\;\mathbf{s^i_n}=[s^i_{n,0},\ldots ,s^i_{n,N}]^T,\;\;\;\mathbf{t^i_n}=[t^i_{n,0},\ldots ,t^i_{n,N}]^T,  
\end{equation*}

Given the fact that all the terms in the
equation~\eqref{bs_euler_lagrange} admits a separated representation, the equation~\eqref{bs_euler_lagrange} can be written in a matrix form.

The following matricial expressions allow us to deduce the matricial form
for the equation~\eqref{bs_euler_lagrange}: 

\begin{equation*}
  \int_{\Omega_1\times\Omega_2\times\Omega_3}(r^i_n\otimes s^i_n\otimes t^i_n)(r\otimes
  s^i_n\otimes t^i_n)=[\mathbf{t^i_n}^TM\mathbf{t^i_n}][\mathbf{s^i_n}^TM\mathbf{s^i_n}]M\mathbf{r^i_n},
\end{equation*}

\begin{equation*}
\small
  \begin{split}
\int_{\Omega_1\times\Omega_2\times\Omega_3}\left(A\nabla (r^i_n\otimes s^i_n\otimes
    t^i_n)\right)\nabla (r\otimes s^i_n\otimes
  t^i_n)&=\left(\frac{\sigma^2_2}{2}[\mathbf{t^i_n}^TM\mathbf{t^i_n}][\mathbf{s^i_n}^TL\mathbf{s^i_n}]+\frac{\sigma^2_3}{2}[\mathbf{t^i_n}^TL\mathbf{t^i_n}][\mathbf{s^i_n}^TM\mathbf{s^i_n}]\right)M\mathbf{r^i_n}
  \\
&+\left(\frac{\rho_{1,2}\sigma_1\sigma_2}{2}[\mathbf{t^i_n}^TM\mathbf{t^i_n}][\mathbf{s^i_n}^TD\mathbf{s^i_n}]+\frac{\rho_{1,3}\sigma_1\sigma_3}{2}[\mathbf{t^i_n}^TD\mathbf{t^i_n}][\mathbf{s^i_n}^TM\mathbf{s^i_n}]\right)
\\
&(D+D^T)\mathbf{r^i_n}+\frac{\sigma^2_1}{2}[\mathbf{s^i_n}^TM\mathbf{s^i_n}][\mathbf{t^i_n}^TM\mathbf{t^i_n}]L\mathbf{r^i_n}, 
  \end{split}
\end{equation*}

\begin{equation*}
\small
  \begin{split}
  \int_{\Omega_1\times\Omega_2\times\Omega_3}\left(a\nabla (r^i_n\otimes s^i_n\otimes
    t^i_n)\right)(r\otimes s^i_n\otimes
  t^i_n)&=[\mathbf{s^i_n}^TM\mathbf{s^i_n}][\mathbf{t^i_n}^TM\mathbf{t^i_n}]B\mathbf{r^i_n}+\left(\frac{\rho_{1,2}\sigma_1\sigma_2}{2}[\mathbf{s^i_n}^TC\mathbf{s^i_n}][\mathbf{t^i_n}^TM\mathbf{t^i_n}]\right. \\
    &\left. +\frac{\rho_{1,3}\sigma_1\sigma_3}{2}[\mathbf{t^i_n}^TC\mathbf{t^i_n}][\mathbf{s^i_n}^TM\mathbf{s^i_n}]\right)D\mathbf{r^i_n}+\left(\frac{\rho_{1,2}\sigma_1\sigma_2}{2}[\mathbf{t^i_n}^TM\mathbf{t^i_n}][\mathbf{s^i_n}^TD\mathbf{s^i_n}]\right. \\
    &\left. +\frac{\rho_{1,3}\sigma_1\sigma_3}{2}[\mathbf{t^i_n}^TD\mathbf{t^i_n}][\mathbf{s^i_n}^TM\mathbf{s^i_n}]\right)C\mathbf{r^i_n}+\left([\mathbf{t^i_n}^TM\mathbf{t^i_n}][\mathbf{s^i_n}^TB\mathbf{s^i_n}]\right.\\
&\left.+\frac{\rho_{2,3}\sigma_2\sigma_3}{2}[\mathbf{t^i_n}^TC\mathbf{t^i_n}][\mathbf{s^i_n}^TD\mathbf{s^i_n}]+[\mathbf{t^i_n}^TB\mathbf{t^i_n}][\mathbf{s^i_n}^TM\mathbf{s^i_n}]\right. \\
&\left.+\frac{\rho_{2,3}\sigma_2\sigma_3}{2}[\mathbf{t^i_n}^TD\mathbf{t^i_n}][\mathbf{s^i_n}^TC\mathbf{s^i_n}]\right)M\mathbf{r^i_n}
\end{split}
  \end{equation*}
\normalsize

\noindent where the matrices $M,L,B,C,D$ are explicitly computable tridiagonal matrices of size $N\times N$,
with $N$ the number of intervals in each direction. The computation of the
components for these matrices boils down to one-dimensional integrals.

In this way, solving~\eqref{bs_euler_lagrange} with a fixed point procedure
allows us to obtain for a fixed $i$ such that $1\leq i\leq N$, the $n$-th term of the sum $\sum^n_{k=1}r^i_k\otimes s^i_k\otimes t^i_k$ which is an approximation of the solution at time $t_i=i\Delta t$ of the problem~\eqref{eq:op_bs_edp_x}.

\section{Numerical results}
\label{section:numerical_results}

\subsection{Testing the method against an analytical solution}
\label{section:test_analytical_sol}

In this part, we apply our greedy algorithm~\eqref{min_pb_bs_imex_greedy} to solve the problem~\eqref{eq:op_bs_edp_x} with the following initial condition:
\begin{equation}
  \label{eq:init_cond_analytic_sol}
  u(0,x_1,\ldots ,x_d)=\left(K-\prod^d_{i=1}\frac{x_i}{(1-x_i)}\right)^+
\end{equation}
for which the solution is analytically known.

Using the Feynman-Kac theorem~\eqref{eq:feynman_kac} we get that the solution of the PDE~\eqref{eq:std_bs_fwd} is given by
$\mathbb{E}\left[e^{-rT}(K-\prod^d_{i=1}S^i_T)^+\right]$, which is
possible to calculate analytically in the Black-Scholes model. We have:

\begin{equation}
\label{eq:pricing_analytic_proof}  
\mathbb{E}\left[e^{-rT}(K-\prod^d_{i=1}S^i_T)^+\right] = 
e^{-rT}K\mathbb{P}\left(K>\prod^d_{i=1}S^i_T\right) -
e^{-rT}\mathbb{E}\left[\prod^d_{i=1}S^i_T\mathbf{1}_{\left\{K>\prod^d_{i=1}S^i_T\right\}}\right] 
\end{equation}

We get the quantity $\mathbb{P}\left(K>\prod^d_{i=1}S^i_T\right)$ as follows:

\begin{eqnarray*}
  \mathbb{P}\left(K>\prod^d_{i=1}S^i_T\right)& = & \mathbb{P}\left(e^{\sum^d_{i=1}X^i_T}<\frac{K}{\prod^d_{i=1}S^i_0}\right) \\
  & = & \mathbb{P}\left(Y<\log\left(\frac{K}{\prod^d_{i=1}S^i_0}\right)\right)
\end{eqnarray*}

where $X^i_T=(r-\frac{\sigma^2_i}{2})T+\sigma_iW^i_T$ and
$Y=\sum^d_{i=1}X^i_T$ is a normal random variable with mean equals to
$\sum^d_{i=1}\left(r-\frac{\sigma^2_i}{2}\right)T$ and variance given by
$\sum^d_{i=1}\sum^d_{j=1}\rho_{ij}\sigma_i\sigma_jT$.

Besides, we remark that:

\begin{equation*}
\mathbb{E}\left[\prod^d_{i=1}S^i_T\mathbf{1}_{\left\{K>\prod^d_{i=1}S^i_T\right\}}\right]=\mathbb{E}\left[e^Y\mathbf{1}_{\{e^Y<\frac{K}{\prod^d_{i=1}S^i_0}\}}\right]\prod^d_{i=1}S^i_0  
\end{equation*}

\noindent where the last term can be calculated analytically. We remark that
the analytic solution~\eqref{eq:pricing_analytic_proof} is not separable with respect to each coordinate.

We present in figure~\ref{analytical_sol_T_2D} a numerical example of the
solution obtained with our algorithm and the analytic solution. In
figure~\ref{analytical_sol_T_2D_cut_plane} we see the same surface but
intersected with the plane $x_1=x_2$.

\begin{figure}[htbp]
 \centering
 \includegraphics[width=0.5\linewidth]{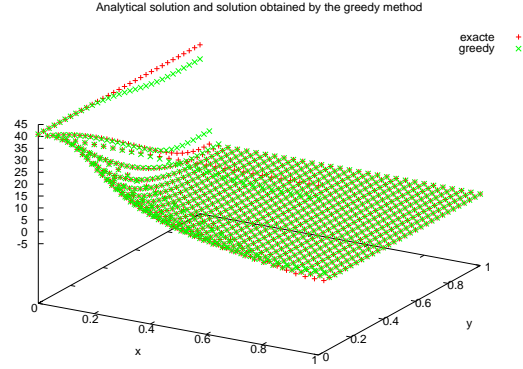} 
 \caption{The analytical solution and the numerical one obtained with our method for the problem~\eqref{eq:std_bs_fwd} with the initial condition~\eqref{eq:init_cond_analytic_sol} in a two-dimensional framework. For this example, we consider $\Delta t=\frac{1}{100}$ and $\Delta x=\frac{1}{30}$.}
\label{analytical_sol_T_2D}
\end{figure}

\begin{figure}[htbp]
 \centering
 \includegraphics[width=0.5\linewidth]{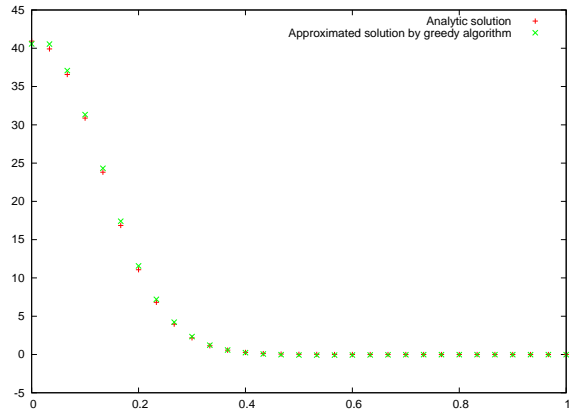} 
 \caption{The analytical solution and the numerical one obtained with our
   method for the problem~\eqref{eq:std_bs_fwd} with the initial
   condition~\eqref{eq:init_cond_analytic_sol} in a two-dimensional
   framework. We represent the intersection between the surface in figure~\ref{analytical_sol_T_2D}
   and the plane $x_1=x_2$. For this example, we consider $\Delta t=\frac{1}{100}$ and $\Delta x=\frac{1}{30}$.}
\label{analytical_sol_T_2D_cut_plane}
\end{figure}

Figure~\ref{cv_curves_analytical_sol} shows the convergence curves, i.e, the $L^2$ relative error with respect to the number of iterations of the algorithm according to the dimension. We note that the number of iterations needed to obtain convergence increases as the dimension of the problem increases.

\begin{figure}[htbp]
 \centering
 \includegraphics[width=0.5\linewidth]{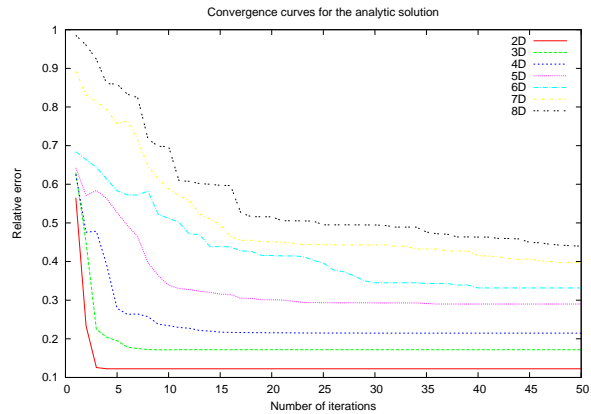} 
 \caption{Convergence curves for the solution at time $T$ of the equation~\eqref{eq:std_bs_fwd} with initial condition given by~\eqref{eq:init_cond_analytic_sol}. To obtain this curves, we consider $\Delta x=0.1$ for each dimension and $\Delta t=0.01$.}
\label{cv_curves_analytical_sol}
\end{figure}

\subsection{Results on the Black-Scholes equation}
\label{section:black_scholes_results}

In this section, we show the results that we obtained applying our greedy algorithm described in the previous section to the Black-Scholes equation.

Figure~\ref{greedy_bs_sol_T_2D} represents the approximation of the solution at time $T$ to the problem~\eqref{eq:std_bs_fwd} with initial condition~\eqref{eq:init_cond} obtained using our greedy algorithm defined by the equations~\eqref{bs_euler_lagrange}

\begin{figure}[htbp]
 \centering
 \includegraphics[width=0.5\linewidth]{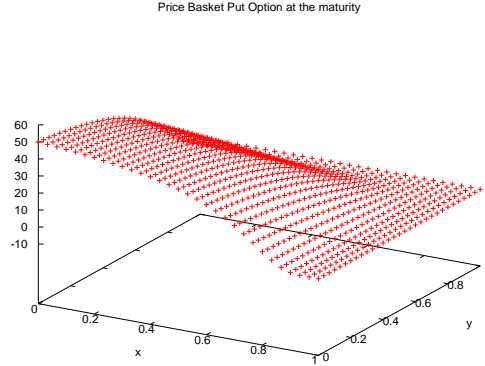} 
 \caption{The approximated solution obtained with our method for the problem~\eqref{eq:std_bs_fwd} with the initial condition~\eqref{eq:init_cond}.}
\label{greedy_bs_sol_T_2D}
\end{figure}

Figure~\ref{4D_prices} represent the price of a basket
put option when all the assets take the same value,$\text{ i.e. }S_1=\ldots
=S_d$. Precisely, Figure~\ref{4D_prices} compares prices obtained with different discretizations $\Delta x$.

\begin{figure}[htbp]
 \centering
 \includegraphics[width=0.5\linewidth]{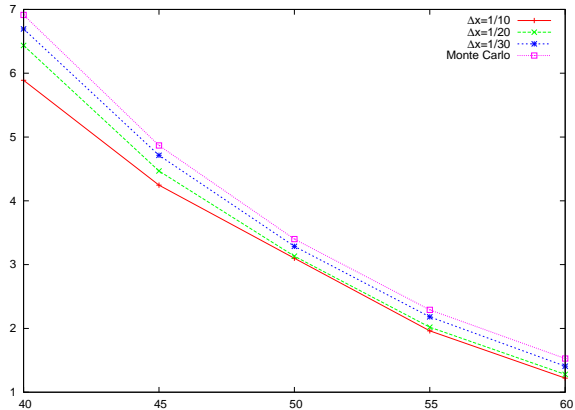} 
 \caption{The approximated solution obtained with our method for the problem~\eqref{eq:std_bs_fwd} with the initial condition~\eqref{eq:init_cond} in a four-dimensional framework. For these results, we set $\Delta x = 0.1,0.2$ and $\Delta t = 0.01$.}
\label{4D_prices}
\end{figure}

In terms of computational time, the greedy
approach~\eqref{min_pb_bs_imex_greedy} is not competitive compared to Monte
Carlo methods when one is interested in the price of only one value of the
spot, but on the other side the curve of prices is obtained for any
time $t\in[0,T]$ and any price spot.  

\subsection{Application as a variance reduction method}
\label{section:variance_reduction}

In this part, we show that we can use the solution obtained by the greedy
method described above in order to find a control variable to reduce the
variance when calculating the price an of option.

We can re-write the equations~\eqref{dyn_bs} and~\eqref{correl_browniens} which
define the Black-Scholes model as follows:

\begin{equation}
  \label{eq:dyn_brown_indep}
  \frac{dS^i_t}{S^i_t}=rdt+\sigma^i\sum^d_{j=1}H_{i,j}dW^j_t
\end{equation}

\noindent where W is a $d$-dimensional Brownian motion and the matrix $H$ verifies $HH^t=\Sigma$ where $\Sigma$ is a $d\times d$ matrix such that $\Sigma_{i,j}=1$ if $j=i$ and $\Sigma_{i,j}=\rho_{ij}$ otherwise.

We recall that the price of a basket put option is given by

\begin{equation*}
  \mathbb{E}\left[e^{-rT}f(S^1_T,S^2_T,\ldots ,S^d_T)\right]
\end{equation*}

\noindent where $f(S^1_T,S^2_T,\ldots ,S^d_T)=\left(K-\frac{1}{d}\sum^d_{i=1}S^i_T\right)_+$.

Now, for the sake of simplicity, let us consider $r=0$. Generalization to
$r\neq 0$ is straightforward. Let us introduce the Kolmogorov equation:
\begin{eqnarray}
  \label{eq:kolmogorov_back}
  \begin{cases}
  \partial_t \hat{P}-\frac{1}{2}A:\nabla^2 \hat{P} = 0 \\
  \hat{P}(0,x)=f(x)
\end{cases}
\end{eqnarray}

\noindent where $A=FH(FH)^T$ and $F$ is a diagonal matrix such that $F_{i,i}=\sigma^iS^i$ for $i=1,\ldots ,d$. Notice that 

$$\hat{P}(T,S_0)=\mathbb{E}\left[f(S^1_T,S^2_T,\ldots ,S^d_T)\right].$$

Therefore, we have

\begin{equation*}
  \hat{P}(0,S_T)-\hat{P}(T,S_0)=\int^T_0FH\;\nabla \hat{P}(T-t,S_t)dB_t
\end{equation*}

\noindent and thus,

\begin{equation}
  \label{eq:price_var_reduct}
  \hat{P}(T,S_0)=f(S_T)-\int^T_0 FH\;\nabla \hat{P}(T-t,S_t)dB_t
\end{equation}

The random variable $Y=\int^T_0 FH\;\nabla \hat{P}(T-t,S_t)dB_t$ has zero mean and
is a perfect control variable since

\begin{equation*}
  \text{Var}\left[f(S_T)-Y\right]=0.
\end{equation*}

As we do not know the solution $\hat{P}$, in practice, we calculate an approximation
$\hat{P}^\star$ of $\hat{P}$ using the greedy algorithm presented in Section~\ref{section:black_scholes_pde_gral}. Therefore, we obtain an
approximated control variable $Y^\star=\int^T_0 FH\;\nabla
\hat{P}^\star(T-t,S_t)dB_t$ and we can compute an approximation of $\hat{P}(T,S_0)$ by
Monte Carlo, computing the following quantity:

\begin{equation*}
 \mathbb{E}\left[f(S_T)-\int^T_0 FH\;\nabla \hat{P}^\star(T-t,S_t)dB_t\right].
\end{equation*}

We remark that this idea can be applied to any payoff function and that for a
new value of $S_0$ we use the same approximation $\hat{P}^\star$.

In Table~\ref{table_var_reduct_1} and~\ref{table_var_reduct_2} we present the
performance of our variance reduction method compared with the variance
obtained with the classical method, i.e. calculating
directly $\mathbb{E}[f(S_T)]$. 

\begin{table}
\begin{center}
\begin{tabular}{|c|c|c|}
  \hline
  Dimension & Without variance reduction & With variance reduction\\
  \hline
  4 &0.1233 &0.0012 \\
  5 &0.1204 &0.0034 \\
  6 &0.1197 &0.0078 \\
  7 &0.1245 &0.0113 \\
  8 &0.1257 &0.0254 \\
  \hline
\end{tabular}
\caption{Variance with a correlation parameter
  $\rho_{i,j}=0.9$ constant between all the assets.\label{table_var_reduct_1}}
\end{center}
\end{table}

\begin{table}
\begin{center}
\begin{tabular}{|c|c|c|}
  \hline
  Dimension & Without variance reduction & With variance reduction\\
  \hline
  4 &0.1256 &0.0023 \\
  5 &0.1248 &0.0045 \\
  6 &0.1230 &0.0096 \\
  7 &0.1199 &0.0158 \\
  8 &0.1232 &0.0296 \\
  \hline
\end{tabular}
\caption{Variance with a correlation parameter $\rho_{i,j}=0.1$
  constant between all the assets.\label{table_var_reduct_2}}
\end{center}
\end{table}

For two typical values of the correlation, we observe that the reduction of
the variance is important, for example, up to a factor 6 in dimension $8$.

\newpage

\bibliography{reference}
\bibliographystyle{plain}
    
\end{document}